\documentclass{amsproc}
\usepackage{amssymb}
\usepackage{amsmath}
\usepackage{graphicx}
\graphicspath{ {./images/} }
\usepackage{manfnt}
\usepackage{amsthm}
\usepackage{stmaryrd}
\usepackage{url}
\usepackage{mathtools}
\usepackage{color}
\usepackage{enumitem}
\usepackage{array}
\usepackage{tikz}
\usepackage{tikz-cd}
\usetikzlibrary{intersections}
\usepackage{mathrsfs}
\numberwithin{equation}{section}
\usepackage{xcolor}
\usepackage{hyperref}

\newcommand{\bbm}{\begin{bmatrix}}
\newcommand{\ebm}{\end{bmatrix}}
\newcommand{\bv}{\begin{vmatrix}}
\newcommand{\ev}{\end{vmatrix}}

\newcommand{\g}{\mathfrak{g}}

\newcommand{\C}{\mathbb{C}}

\newcommand{\mf}{\mathfrak}
\newcommand{\mc}{\mathcal}

\newcommand{\Z}{\mathbb{Z}}

\newcommand{\R}{\mathbb{R}}
\newcommand{\PP}{\mathbb{P}}

\newcommand{\bp}{\begin{pmatrix}}
\newcommand{\ep}{\end{pmatrix}}

\DeclareMathOperator{\Lie}{Lie}

\DeclareMathOperator{\stab}{stab}

\DeclareMathOperator{\SL}{SL}
\DeclareMathOperator{\SU}{SU}

\newtheorem{theorem}{Theorem}[section]

\theoremstyle{definition}

\theoremstyle{remark}
\newtheorem{remark}[theorem]{Remark}

\numberwithin{equation}{section}

\title{Four examples of Beilinson--Bernstein localization}

\author{Anna Romanov}
\address{A14 Quadrangle, Sydney Mathematical Research Institute, The University of Sydney, New South Wales, Australia 2006}
\email{anna.romanov@sydney.edu.au}
\thanks{Supported by the National Science Foundation Award No. 1803059}


\date{today}

\begin{document}
\maketitle

\begin{abstract} Let $\mf{g}$ be a complex semisimple Lie algebra. The Beilinson--Bernstein localization theorem establishes an equivalence of the category of $\mf{g}$-modules of a fixed infinitesimal character and a category of modules over a twisted sheaf of differential operators on the flag variety of $\mf{g}$. In this expository paper, we give four detailed examples of this theorem when $\mf{g}=\mf{sl}(2,\C)$. Specifically, we describe the $\mc{D}$-modules associated to finite-dimensional irreducible $\mf{g}$-modules, Verma modules, Whittaker modules, discrete series representations of $\SL(2,\R)$, and principal series representations of $\SL(2,\R)$. 
\end{abstract}

\section{Introduction}
This paper revolves around the following beautiful theorem of Beilinson--Bernstein. Let $\mf{g}$ be a complex semisimple Lie algebra, $\mc{U}(\mf{g})$ its universal enveloping algebra, and $\mc{Z}(\mf{g}) \subset \mc{U}(\mf{g})$ the center. Fix a Cartan subalgebra $\mf{h} \subset \mf{g}$ and let $\lambda \in \mf{h}^*$. The Weyl group orbit $\theta \subset \mf{h}^*$ of $\lambda$ determines an infinitesimal character $\chi_\theta:\mc{Z}(\mf{g})\rightarrow \C$, and we denote by $\mc{U}_\theta$ the quotient of $\mc{U}(\mf{g})$ by the ideal generated by the kernel of $\chi_\theta$. In \cite{BB}, Beilinson--Bernstein construct a twisted sheaf of differential operators $\mc{D}_\lambda$ on the flag variety $X$ of $\mf{g}$ associated to $\lambda$. Denote by $\mc{M}(\mc{U}_\theta)$ the category of $\mc{U}_\theta$-modules and by $\mc{M}_{qc}(\mc{D}_\lambda)$ the category of quasi-coherent $\mc{D}_\lambda$-modules. 
\begin{theorem}
\label{localization}
{\em (Beilinson--Bernstein \cite{BB})} Let $\lambda \in \mf{h}^*$ be dominant and regular. There is an equivalence of categories
\[
\begin{tikzcd}
\mc{M}(\mc{U}_\theta) \arrow[r, bend left=30, "\Delta_\lambda"]
& \mc{M}_{qc}(\mc{D}_\lambda)
\arrow[l, bend left=30, "\Gamma"]
\end{tikzcd}
\]
given by the localization functor $\Delta_\lambda(V)=\mc{D}_\lambda \otimes _{\mc{U}_\theta} V$ and the global sections functor $\Gamma$. 
\end{theorem}
Theorem \ref{localization} lets us transport the study of representations of $\mf{g}$ to the setting of $\mc{D}_\lambda$-modules, where local techniques of algebraic geometry can be employed. It is difficult to overstate the impact of this theorem on modern representation theory. Starting with its initial use to prove Kazhdan--Lusztig's conjecture that composition multiplicities of Verma modules are given by Kazhdan--Lusztig polynomials, this theorem fundamentally changed the way that questions in representation theory are approached. 

The aim of this paper is to provide concrete examples of this powerful theorem for the simplest nontrivial example, $\mf{g}=\mf{sl}(2,\C)$. We give explicit descriptions of four types of $\mc{D}_\lambda$-modules, then we compute the $\mf{g}$-module structure on their global sections to realize them as familiar representations of $\mf{g}$. Our first example (Section \ref{Finite dimensional g-modules}) of finite-dimensional $\mf{g}$-modules illustrates the classical Borel-Weil theorem for $\mf{sl}(2,\C)$. In our second example (Section \ref{Verma modules}), we realize the global sections of standard Harish-Chandra sheaves as Verma modules or dual Verma modules. Our third example (Section \ref{Admissible representations of SL(2,R)}) describes the $\mc{D}_\lambda$-modules corresponding to discrete series and principal series representations of $\SL(2,\R)$. This example illustrates two families of representations which arise in the classification of irreducible admissible representations of a real reductive Lie group via Harish-Chandra sheaves. General descriptions of this classification and its applications are discussed in \cite{BB, duality, LV, intertwining, Vogan3}. Our final example (Section \ref{Whittaker modules}) is of a $\mc{D}_\lambda$-module whose global sections have the structure of a Whittaker module. This $\mc{D}_\lambda$-module is an example of a twisted Harish-Chandra sheaf. Whittaker modules were first introduced by Kostant in \cite{Kostant}, and a geometric approach to studying them using $\mc{D}_\lambda$-modules was developed in \cite{TwistedSheaves} and used in \cite{Romanov} to establish the structure of their composition series. 

\section{Set-up}
\label{Set-up}

For the remainder of this paper, let $\mf{g}=\mf{sl}(2,\C)$ with the standard basis 
\[
 E=\bp 0 & 1 \\ 0 & 0 \ep,\hspace{2mm} H = \bp 1 & 0 \\ 0 & -1 \ep, \hspace{2mm} F = \bp 0 & 0 \\ 1 & 0 \ep.  
\]
Then $\mf{g}$ has a triangular decomposition $\mf{g} = \overline{\mf{n}} \oplus \mf{h} \oplus \mf{n}$, where $\overline{\mf{n}} = \C F, \mf{h}=\C H,$ and $\mf{n} = \C E$. Let $\mf{b}=\mf{h} \oplus \mf{n}$ be the upper triangular Borel subalgebra. Denote the corresponding complex Lie groups by $G=\SL(2,\C)$,
\[
 B=\left\{ \left.\bp a & b \\ 0 & a^{-1} \ep \right\vert a\in \C^\times, b \in \C \right\},  \text{ and } N=\left\{ \left. \bp 1 & b \\ 0 & 1 \ep \right\vert b \in \C \right\}. 
\]

Let $\Sigma^+ \subset \Sigma$ be the associated set of positive roots in the root system of $\mf{g}$. Denote the single element of $\Sigma^+$ by $\alpha$, and by $\alpha^\vee$ the corresponding coroot. Let $P(\Sigma)$ be the weight lattice in $\mf{h}^*$. The Weyl group $W$ of $\mf{g}$ is isomorphic to $\Z / 2\Z$. Let $X=G/B$ be the flag variety of $\mf{g}$. The variety $X$ is isomorphic to $\C \PP ^1$, which we identify with the set of lines through the origin in $\C^2$. (Indeed, the natural action of $G$ on $\C^2$ by matrix multiplication gives a transitive action of $G$ on the set of lines through the origin in $\C^2$, and the stabilizer of the line spanned by $(1,0)$ is $B$.) Denote the line through the point $(x_0, x_1) \in \C^2\backslash \{(0,0)\}$ by $[x_0:x_1]$, and define a map
\begin{equation}
    \label{p}
    p: \C^2 \backslash \{(0,0)\} \rightarrow X, \hspace{2mm} (x_0, x_1) \mapsto [x_0, x_1].
\end{equation}
The group $G$ acts on $X$ via the action 
\begin{equation}
    \label{G action on X}
\bp a & b \\ c & d \ep \cdot [x_0: x_1] = [ax_0 + bx_1 : cx_0 + dx_1].
\end{equation}
We distinguish the points $[0:1]$ and $[1:0]$ by labeling them 
\[
\infty=[0:1], \text{ and } 0=[1:0]. 
\]

An open cover of $X$ is given by $\{U_0, U_\infty\}$, where $U_0=\C\PP^1 - \{\infty\}$, $U_\infty = \C\PP^1 - \{0\}$. Denote by $V=U_0 \cap U_ \infty \simeq \C^\times$. We identify $U_0$ with $\C$ using the coordinate $z:U_0 \rightarrow \C$ given by $z([x_0:x_1])=x_1/x_0$ and $U_\infty$ with $\C$ via the coordinate $w([x_0:x_1]) = x_0/x_1$. On $V$, the coordinates are related by 
\begin{equation}
    \label{coordinate relation}
w = \frac{1}{z}. 
\end{equation}

Let $\mc{O}_X$ be the structure sheaf of $X$, and $\mc{D}_X$ the sheaf of differential operators. For an affine subset $U \subset X$, we denote by $R(U)$ its ring of regular functions and $D(U)$ the ring of differential operators on $R(U)$\footnote{We align our notation with \cite{d-modulesnotes}, and encourage the reader interested in more details on the theory of algebraic D-modules to consult this excellent reference.}. A $\C$-linear endomorphism $T:\mc{V} \rightarrow \mc{V}$ of an $\mc{O}_X$-module $\mc{V}$ is a {\em differential endomorphism of order $\leq n$ } if for any open set $U\subset X$ and $(n+1)$-tuple of regular functions $p_1, \ldots, p_n$ in $R(U)$, we have $[\ldots[[T,p_0],p_1],\ldots,p_n]=0$ on $U$. This generalizes the notion of a differential operator on $X$. Denote by $\mathcal{D}\mathit{iff}(\mc{V},\mc{V})$ the sheaf of differential endomorphisms of $\mc{V}$.

\section{Serre's twisting sheaves}
\label{Serre's twisting sheaves}

We start by introducing a family of sheaves of $\mc{O}_X$-modules on $X$ parameterized by the integers. These sheaves will eventually provide our first examples of $\mc{D}_\lambda$-modules in Section \ref{Finite dimensional g-modules}. 

Let $\mc{L}$ be an invertible $\mc{O}_X$-module. Then there exist isomorphisms $\varphi_0:\mc{L}|_{U_0} \xrightarrow{\sim} \mc{O}_{U_0}$ and $\varphi_{\infty}: \mc{L}|_{U_\infty} \xrightarrow{\sim} \mc{O}_{U_\infty}$. By restricting $\varphi_0$ and $\varphi_\infty$, we obtain two isomorphisms of $\mc{L}|_{V}$ with $\mc{O}_V$. Combining these we obtain an $\mc{O}_V$-module isomorphism
$\varphi:\mc{O}_V \rightarrow \mc{O}_V.$
Taking global sections results in an $R(V)$-module isomorphism $\varphi:R(V) \rightarrow R(V)$ which we call by the same name. Specifically, $\varphi$ is the $R(V)$-module isomorphism making the following diagram commute.
\[
\begin{tikzcd}
R(V)  \arrow[rightarrow, r, "\varphi"] \arrow[rightarrow, d, "\varphi_\infty^{-1}"'] & R(V)\\
\mc{L}(V) \arrow[rightarrow, r, "id"] & \mc{L}(V) \arrow[rightarrow, u, "\varphi_0"']
\end{tikzcd}
\]

Because $1$ generates $R(V)$ as an $R(V)$-module, this morphism is completely determined by the image of $1$; that is, if $\varphi(1)=p \in R(V)$, then $\varphi(q)=qp$ for any $q \in R(V)$. The morphism $\varphi$ is an isomorphism, so $\varphi^{-1}$ is also given by multiplication by a regular function $r=\varphi^{-1}(1)$, and $rp=1$. Hence $r$ and $p$ have no zeros or poles in $V$, so in the coordinate $z$, they must be of the form 
$p(z)=cz^n$ and $r(z)=\frac{1}{c} z^{-n}$ for some $c \in \C^\times$ and $n \in \Z$. We conclude that the transition function of $\mc{L}$ is of the form 
\begin{equation}
    \label{varphi} 
    \varphi: R(V) \rightarrow R(V), \hspace{2mm} 1 \mapsto cz^n.
\end{equation}

The integer $n$ determines the sheaf $\mc{L}$ up to isomorphism \cite[Cor 6.17]{Hartshorne}, so without loss of generality we can assume $c=1$. We denote the invertible sheaf corresponding to $n \in \Z$ by $\mc{O}(n)$. These are {\em Serre's twisting sheaves}. 

Next we'd like to compute the global sections of $\mc{O}(n)$. Because the transition function $\varphi$ is given by multiplication by $z^n$ and $\mc{O}(n)(U_\infty) \simeq \C[w]$ and $\mc{O}(n)(U_0)\simeq \C[z]$, a global section of $\mc{O}(n)$ is a polynomial $q(w)\in \C[w]$ and a polynomial $p(z) \in \C[z]$ such that $q(1/z)z^n=p(z)$. We can see that a pair of such polynomials only exists when $n \geq 0$, and the polynomial $q(w)$ must be of degree less than or equal to $n$. In this way, the space of global sections of $\mc{O}(n)$ for $n\geq 0$ can be identified with the vector space of polynomials of degree $\leq n$. In particular, 
\[
\dim \Gamma(\mc{O}(n), X) = \begin{cases} n+1 & \text{ if } n \geq 0; \\ 0 & \text{ if } n<0.
\end{cases}
\]

In the computations that follow, it will be useful to have a more explicit realization of the sheaf $\mc{O}(n)$ as a sheaf of homogeneous holomorphic functions on $\C^2 \backslash \{(0,0)\}$. We will describe this realization now. Fix $n \in \Z$, and let $\mc{F}_n$ be the sheaf on $X$ defined by 
\[
\mc{F}_n(U) = \begin{array}{c} \text{vector space of homogeneous}
\\
\text{holomorphic functions on }\\
\text{$p^{-1}(U)$ of degree $n$}
\end{array}
\]
for $U \subset X$ open, where $p$ is the map (\ref{p}). This is an invertible sheaf on $X$. Indeed, for our open cover $\{U_0, U_\infty\}$, we have isomorphisms 
\begin{equation}
    \label{varphi_0}
    \varphi_0:\mc{F}_n(U_0) \xrightarrow{\sim} \mc{O}_X(U_0), \hspace{2mm} f \mapsto p(z):=f(1,z)
\end{equation}
and 
\begin{equation}
    \label{varphi_infty}
    \varphi_\infty: \mc{F}_n(U_\infty) \xrightarrow{\sim} \mc{O}_X(U_\infty), \hspace{2mm} h \mapsto q(w):=h(w,1)
\end{equation}
The invertible sheaf $\mc{F}_n$ must be isomorphic to one of Serre's twisting sheaves. A quick computation using (\ref{varphi_0}) and (\ref{varphi_infty}) shows that the transition function $\varphi= \varphi_0 \circ \varphi_\infty^{-1}$ maps $q \mapsto z^n q$, so $\mc{F}_n \simeq \mc{O}(n)$. 

\section{Twisted sheaves of differential operators}
\label{Twisted sheaves of differential operators}

In \cite{BB}, Beilinson--Bernstein construct a sheaf $\mc{D}_\lambda$ of rings on $X$ for each $\lambda \in \mf{h}^*$ which is locally isomorphic to the sheaf of differential operators on $X$. If $\lambda$ is in the weight lattice, then the sheaf $\mc{D}_\lambda$ can be realized very explicitly. For concreteness, we will work in this setting. Full details of the general construction can be found in \cite[Ch. 2, \S 1]{localization}. 

Let $\lambda \in P(\Sigma)$, so $t:=\alpha^\vee(\lambda) \in \Z$. Define $\mc{D}_t$ to be the sheaf of differential endomorphisms of the $\mc{O}_X$-module $\mc{O}(t-1)$. That is, in the notation of \S\ref{Set-up}, 
\[
\mc{D}_t = \mc{D} \mathit{iff}(\mc{O}(t-1), \mc{O}(t-1)). 
\]
\begin{remark}
The $-1$ in this definition is a rho shift: if $\rho=\frac{1}{2} \alpha$, then $\alpha^\vee(\lambda - \rho)=t-1$. In general, if $X$ is the flag variety of a semisimple Lie algebra $\mf{g}$, $\rho$ is the half sum of positive roots, and $\lambda \in P(\Sigma)$, then $\mc{D}_\lambda = \mc{D}\mathit{iff}(\mc{O}(\lambda - \rho), \mc{O}(\lambda - \rho))$. 
\end{remark}

The sheaf $\mc{D}_t$ is locally isomorphic to $\mc{D}_X$, so it is an example of a  twisted sheaf of differential operators \cite[Ch. 1, \S 1]{localization}. 

Associated to $\mc{D}_t$ is a ring isomorphism 
\[
\psi:D(V) \rightarrow D(V),
\]
with the property that for $T \in D(V)$, the diagram 
\[
\begin{tikzcd}
R(V) \arrow[rightarrow, r, "\varphi"] \arrow[rightarrow, d, "T"']& R(V) \arrow[rightarrow, d, "\psi(T)"] \\
R(V) \arrow[rightarrow, r, "\varphi"] & R(V) 
\end{tikzcd}
\]
commutes. Here $\varphi(q)=z^{t-1}q$, as in equation (\ref{varphi}). This isomorphism completely determines the data of the sheaf $\mc{D}_t$, so we would like to compute it explicitly. Since $V \simeq \C^\times$, the ring $D(V)$ is generated by multiplication by $z$ and differentiation with respect to $z$, which we denote by $z$ and $\partial_z$, respectively, and $\psi$ is determined by the image of $z$ and $\partial_z$. We can see that for $q \in R(V)$,
\[
\varphi(z \cdot q) = z^{t-1}zq = zz^{t-1}q=z \cdot \varphi(q),
\]
so $\psi(z)=z$. Because $\psi$ is an isomorphism of rings of differential operators, it must preserve order, so $\psi(\partial_z) = a \partial_z + b$ for some $a,b \in R(V)$. An easy computation shows that $a=1$ and $b=-(t-1)/z$. Hence the sheaf $\mc{D}_t$ is determined up to isomorphism by the ring isomorphism 
\begin{equation}
    \label{psi} \psi: D(V) \rightarrow D(V), \hspace{2mm}  z \mapsto z, \hspace{2mm}  \partial_z \mapsto \partial_z - \frac{t-1}{z}. 
\end{equation}
One can check the equality 
\begin{equation}
    \label{conjugation to subtraction}
    \partial_z - \frac{t-1}{z}=z^{t-1}\partial_z z^{-(t-1)}
\end{equation}
of differential operators. 
\begin{remark}
If we allow $t \in \C$ to be arbitrary, equation (\ref{psi}) still defines a ring isomorphism. In contrast, the equation (\ref{varphi}) in Section \ref{Serre's twisting sheaves} only defines an $R(V)$-module isomorphism for $n \in \Z$. This reflects the fact that the sheaves $\mc{O}(n)$ are only defined for integral values of $n$, whereas we can extend the definition of $\mc{D}_t$ given above to non-integral values of $t$ by defining a sheaf of rings with gluing given by $\psi$. However, for $t \in \C \backslash \Z$, the sheaves $\mc{D}_t$ can no longer be realized as sheaves of differential endomorphisms of an $\mc{O}_X$-module. 
\end{remark}

We would like to realize global sections of $\mc{D}_t$-modules as $\mf{g}$-modules. To do this, we need to give $\mc{D}_t$ the extra structure of a {\em homogeneous} twisted sheaf of differential operators. This extra structure consists of a $G$-action on $\mc{D}_t$ and an algebra homomorphism $\beta: U(\g) \rightarrow \Gamma(X, \mc{D}_t)$ satisfying some compatibility conditions \cite[Ch. 1 \S 2]{localization}. This additional structure arises naturally if we use the explicit realization of $\mc{O}(t-1)$ as a sheaf of homogeneous holomorphic functions given in Section \ref{Serre's twisting sheaves}. 

There is a natural action of $G$ on $\mc{O}(t-1)$ given by 
\begin{equation*}
    \label{G-action}
    g \cdot f(x_0, x_1) = f (g^{-1} \cdot (x_0, x_1)),
\end{equation*}
where $g^{-1} \cdot (x_0, x_1)$ is the standard action of $G$ on $\C^2$ by matrix multiplication. We can use this to define an action of $G$ on $\mc{O}_X$ which makes the isomorphisms $\varphi_0$ and $\varphi_\infty$ (equations (\ref{varphi_0}) and (\ref{varphi_infty})) $G$-equivariant. Explicitly, if 
\[
g=\bp a & b \\ c & d \ep \in \SL(2,\C) \text{ and } p(z) \in \mc{O}_X(U_0), q(w) \in \mc{O}_X(U_\infty), 
\]
this action is given by 
\begin{equation}
\label{U_0 action}
g \cdot p(z) = (d-bz)^{t-1}p\left(\frac{-c+az}{d-bz}\right) 
\end{equation}
in the $U_0$ chart and 
\begin{equation}
    \label{U_infty action}
    g \cdot q(w)=(-cw + a) ^{t-1}q\left(\frac{dw-b}{-cw+a}\right)
\end{equation}
in the $U_\infty$ chart. Note that for $t \neq 1$, this differs from the standard action of $G$ on $\mc{O}_X$. 

One of the compatibility conditions of a homogeneous twisted sheaf of differential operators is that the $\mf{g}$-action obtained by differentiating the $G$-action on $\mc{D}_t$ should agree with the $\mf{g}$-action obtained from the map $\beta$. Hence we can compute the local $\mf{g}$-action given by $\mc{D}_t$ by differentiating the $G$-actions (\ref{U_0 action}) and (\ref{U_infty action}). Explicitly, if $X \in \mf{g}$ and $p \in \mc{O}(t-1)(U_i)$ for $i=0, \infty$, 
\begin{equation}
    \label{g action}
    X \cdot p = \left.\frac{d}{dr}\right\vert_{r=0} (\exp{rX}) \cdot p,
\end{equation}
where $(\exp{rX}) \cdot p$ is the $G$-action on $\mc{O}_X$ given by (\ref{U_0 action}) and (\ref{U_infty action}). Using (\ref{g action}), we can associate differential operators to the Lie algebra basis elements $E,F,H$ in each chart. As an example, we will include the calculation for $E$ in the $U_0$ chart:
\begin{align*}
    E \cdot p(z) &= \left. \frac{d}{dr}\right\vert_{r=0} (1-rz)^{t-1}p\left(\frac{z}{1-rz}\right)\\
    &=\left.z^2(1-rz)^{t-3}p'\left(\frac{z}{1-rz}\right) - z(t+1)(1-rz)^{t-2}p\left(\frac{z}{1-rz}\right) \right\vert_{r=0} \\
    &= z^2p'(z) - z(t-1) p(z). 
\end{align*}
We conclude from this calculation that as a differential operator on $O_X(U_0)$, the Lie algebra element $E$ acts in the coordinate $z$ as 
\[
E_0=z^2 \partial_z - z(t-1).
\]
Similar computations result in the following formulas. On the chart $U_0$, we have 
\begin{align}
\label{E_0}
    E_0 &=z^2 \partial_z - z(t-1), \\
\label{F_0}
    F_0 &= - \partial_z, \text{ and } \\ 
\label{H_0}
    H_0 &= 2z \partial_z - (t-1). 
\end{align}
On the chart $U_\infty$ we have 
\begin{align}
\label{E_infty}
    E_\infty &= -\partial_w, \\
\label{F_infty}
    F_\infty &= w^2 \partial_w - w(t-1), \text{ and } \\
\label{H_infty}
    H_\infty &= -2w \partial_w + (t-1), 
\end{align}
where $\partial_w$ is differentiation with respect to the coordinate $w$. 

The ring homomorphism $\psi$ given in equation (\ref{psi}) relates these formulas on the intersection $V$. We include this computation for $E$ as a sanity check, and encourage the suspicious reader to check the other Lie algebra basis elements. First note that by the relationship (\ref{coordinate relation}) between the coordinates $z$ and $w$, we have 
\begin{equation}
    \label{partial relationship}
\partial_w = -z^2 \partial_z. 
\end{equation}
Using (\ref{partial relationship}) and (\ref{psi}), we can compute the image of the coordinate $w$ and the derivation $\partial_w$ under $\psi$: $\psi(w)=w$ and $\psi(\partial_w)=w^{-(t-1)}\partial_ww^{t-1}$. Finally, we compute: 
\begin{align*}
    \psi(E_\infty)&=\psi(-\partial_w) \\
    &=\psi(-z^2\partial_z) \\
    &=z^2\left(\partial_z - \frac{t-1}{z} \right)\\
    &=z^2 \partial_z - z(t-1)\\
    &=E_0. 
\end{align*}
The middle equality follows from (\ref{conjugation to subtraction}). 

Using the formulas (\ref{E_0}) - (\ref{H_infty}), we can explicitly describe the $\mf{g}$-module structure on the global sections of $\mc{D}_t$-modules. The pairs $(E_0, E_\infty)$, $(F_0, F_\infty)$, and $(H_0, H_\infty)$ each define a global section of $\mc{D}_t$; in particular, they are the global sections which are the images of $E,F$, and $H$ under the map $\beta:U(\mf{g})\rightarrow \Gamma(X, \mc{D}_t)$ which gives $\mc{D}_t$ the structure of a homogeneous twisted sheaf of differential operators. In the remaining sections, we describe several families of $\mc{D}_t$-modules and use these formulas to realize their global sections as familiar $\mf{g}$-modules. 

\section{Finite dimensional modules}
\label{Finite dimensional g-modules}
With the computations of Section \ref{Twisted sheaves of differential operators}, we can realize global sections of $\mc{D}_t$-modules as $\mf{g}$-modules. To warm up, we will do so for the $\mc{D}_t$-module $\mc{O}(t-1)$. Fix $t-1 \in \Z_{\geq 0}$. 

Recall from our discussion in Section \ref{Serre's twisting sheaves} that a global section of $\mc{O}(t-1)$ is a pair $q(w) \in \C[w]$, $p(z) \in \C[z]$ of polynomials such that 
\begin{equation}
    \label{twist}
    p(z)=z^{t-1}q(1/z).
\end{equation}
The polynomial $q(w)$ in any such pair must have degree less than or equal to $t-1$, so $q \in U:= \text{span} \{ 1, w, w^2, \ldots, w^{t-1}\}$. A choice of polynomial $q \in U$ uniquely determines $p(z) \in \C[z]$ satisfying equation (\ref{twist}), so we can identify $\Gamma(X, \mc{O}(t-1))$ with $U$. We can describe the $\mf{g}$-module structure of $\Gamma(X, \mc{O}(t-1))$ by computing the action of the differential operators $E_\infty,F_\infty,H_\infty$ on a basis of $U$ using the formulas (\ref{E_infty}) - (\ref{H_infty}). We choose the basis 
\begin{equation}
    \label{basis}
\left\{u_k:=(-1)^kw^k\right\}_{k=0, \ldots , t-1}
\end{equation}
of $V$ for reasons which will soon become apparent. The action of $E_\infty, F_\infty$, and $H_\infty$ on the basis (\ref{basis}) is given by the formulas
\begin{align}
    \label{E_infty action}
    {\color{red} E_\infty} \cdot u_k &= ku_{k-1}, \\
    \label{F_infty action}
    {\color{blue} F_\infty} \cdot u_k &= ((t-1)-k)u_{k+1}, \\
    \label{H_infty action} 
    {\color{cyan} H_\infty} \cdot u_k &= ((t-1)-2k)u_k. 
\end{align}
We can capture the $\mf{g}$-module structure given by the formulas (\ref{E_infty action}) - (\ref{H_infty action}) with the picture in Figure \ref{finite dimensional picture}. In this picture, a colored arrow indicates that the corresponding differential operator sends the vector space basis element at the start of the arrow to a scalar multiple of the vector space basis element at the end of the arrow. The scalar is given in the label of the arrow. For example, the red arrow labeled $2$ represents the relationship $E_\infty \cdot u_2 = 2u_1$. This picture completely describes the $\mf{g}$-module structure of $\Gamma(X, \mc{O}(t-1))$. It is clear that with this $\mf{g}$-module structure, $\Gamma(X, \mc{O}(t-1))$ is isomorphic to the irreducible $t$-dimensional $\mf{g}$-module of highest weight $t-1$. This illustrates the classical Borel-Weil theorem for $\mf{sl}(2,\C)$.
\begin{figure}
    \centering
\begin{tikzpicture}
\matrix(m)[matrix of math nodes,
row sep=3em, column sep=2.5em,
text height=1.5ex, text depth=0.25ex]
{u_{t+1}&{u_t}&\cdots&{u_2}&{u_1}&u_0 &\\};
\path[->,font=\scriptsize]
(m-1-1) edge [bend right, red] node[below] {$t-1$} (m-1-2)
(m-1-2) edge [bend right, blue]  node[above] {$1$} (m-1-1)
(m-1-3) edge [bend right, red] node[below] {$3$}   (m-1-4) 
edge [bend right, blue] node[above] {$2$}  (m-1-2)
(m-1-2) edge [bend right, red] node[below] {$t-2$}   (m-1-3)
(m-1-2) edge [loop above, cyan] node[above] {$-t+3$} (m-1-2)
(m-1-1) edge [loop above, cyan] node[above] {$-t+1$} (m-1-1)
(m-1-4) edge [bend right, red]  node[below] {$2$}  (m-1-5)
edge [bend right, blue] node[above] {$t-3$}  (m-1-3)
(m-1-5) edge[bend right, red]  node[below] {$1$}  (m-1-6)
edge [bend right, blue]  node[above] {$t-2$} (m-1-4)
(m-1-5) edge [loop above, cyan] node[above] {$t-3$} (m-1-5)
(m-1-4) edge [loop above, cyan] node[above] {$t-5$} (m-1-4)
(m-1-6) edge [bend right, blue] node[above] {$t-1$} (m-1-5)
(m-1-6) edge [loop above, cyan] node[above] {$t-1$} (m-1-6);
\end{tikzpicture}
    \caption{Action of ${\color{red} E_\infty}$, ${\color{blue} F_\infty}$, and $\color{cyan} H_\infty$ on $\mc{O}(t-1)(U_\infty)$}
    \label{finite dimensional picture}
\end{figure}
 
\begin{theorem}
(Borel-Weil) Let $t \in \Z_{\geq 1}$. As $\mf{g}$-modules,
\[
\Gamma(X, \mc{O}(t-1)) \simeq L(t-1) 
\]
where $L(t-1)$ is the irreducible finite-dimensional representation of $\mf{g}$ of highest weight $t-1$. 
\end{theorem}

\begin{remark} In the arguments above, we have done our computations in the chart $U_\infty$. However, we would have arrived at the same conclusion by working in the other chart. Applying $\varphi$ to the basis $\{u_k\}$ results in a basis $\{w_k\}$ of the vector space $W$ of polynomials in $z$ of degree at most $t-1$. Explicitly, $\varphi(u_k)=(-1)^kz^{t-1-k}=:w_k$. Computing the action of the differential operators $E_0, F_0, H_0$ (equations (\ref{E_0}) - (\ref{H_0})) on the basis $\{w_k\}$ results in an identical picture to the one above. 
\end{remark}

\section{Verma modules}
\label{Verma modules}

A natural source of $\mc{D}_t$-modules on the homogeneous space $X$ is a stratification of $X$ by orbits of a group action. In this section, we will describe two $\mc{D}_t$-modules which are constructed using the stratification of $X$ by $N$-orbits, where 
\[
N= \left \{ \left. \bp 1 & b \\ 0 & 1 \ep \right\vert b \in \C \right\}
\]
is the unipotent subgroup of $G$ from Section \ref{Set-up}. These modules are the standard Harish-Chandra sheaves associated to the Harish-Chandra pair $(g,N)$. We will see that their global sections have the structure of highest weight modules for $\mf{g}$. 

The group $N$ acts on $X$ by the restriction of the action (\ref{G action on X}). There are two orbits: the single point $0$ and the open set $U_\infty$. 
\[
\begin{tikzpicture}[scale=0.4]
  \tikzset{
    dot/.style={circle,inner sep=0.15pt,fill,label={\tiny #1},name=#1},
  }
   \begin{scope}[xshift=-5cm]
\shade[ ball color = red!40, opacity = 0.4] (0,0) circle (2);
\draw[opacity=0, name path=P2] (0,0)--(0,2); \draw[] (0,0) circle (2);
\draw[] (-2,0) arc (180:360:2 and 0.8); \draw[dashed] (2,0) arc
(0:180:2 and 0.8); \draw[name path=P1, rotate=100] (-2,0) arc
(180:360:2 and 1); \draw[rotate=100,dashed] (2,0) arc (0:180:2 and 1);
\path [name intersections={of=P1 and P2,by={N}}]; \node[dot] (N) at
(N) {};
\end{scope}

\draw[fill] (0,0) circle (0.2 and 0.08);
\node at (-1.5,0) {\( = \)};
\node at (1.5,0) {\( \sqcup \)};
\node at (0,-2) {\( 0 \)};
\node at (-5,-3) {\( X \)};
\node at (6,-3) {\( U_{\infty} \)};

 \begin{scope}[xshift=5cm]
\shade[ ball color = red!40, opacity = 0.4] (0,0) circle (2);
\draw[opacity=0, name path=P2] (0,0)--(0,2); \draw[] (0,0) circle (2);
\draw[] (-2,0) arc (180:360:2 and 0.8); \draw[dashed] (2,0) arc
(0:180:2 and 0.8); \draw[name path=P1, rotate=100] (-2,0) arc
(180:360:2 and 1); \draw[rotate=100,dashed] (2,0) arc (0:180:2 and 1);
\path [name intersections={of=P1 and P2,by={N}}];  \draw[] (N) circle (0.2 and 0.08); \fill[white] (N) circle
(0.2 and 0.08);
\end{scope}
      \end{tikzpicture}
      \]

For each $O \in \{0, U_\infty\}$, let $i_O:O \hookrightarrow X$ be inclusion. We can construct a $\mc{D}_t$-module associated to each orbit $O$ by using the $\mc{D}_t$-module direct image functor to push forward the structure sheaf:
\begin{equation}
    \label{standard HC sheaf}
    \mc{I}_O:=i_{O+}(\mc{O}_O). 
\end{equation}
Because each orbit $O$ is a homogeneous space for $N$, the sheaf $\mc{O}_O$ has a natural $N$-action. This $N$-action is compatible with the $\mc{D}_t$-action, in the sense that the differential agrees with the $\mf{n}$-action coming from the map $\beta:U(\mf{g})\rightarrow \Gamma(X, \mc{D}_t)$ when restricted to $O$, so $\mc{O}_O$ has the structure of an $N$-homogeneous connection\footnote{Irreducible $N$-homogeneous connections on $O$ are parameterized by irreducible representations of the component group of $\stab_N{x}$ for $x \in O$. In this case, $\stab_N{x}=1$, so $\mc{O}_O$ is the only $N$-homogeneous connection on $O$. In Section \ref{Admissible representations of SL(2,R)} we will see an example of a nontrivial homogeneous connection on an orbit.} on $O$. The $\mc{D}_t$-module $\mc{I}_O$ also carries a compatible $N$-action, so it is a Harish-Chandra sheaf (see \cite[\S 3]{D-modules} for a precise definition) for the Harish-Chandra pair $(\mf{g},N)$. Hence its global sections $\Gamma(X, \mc{I}_O)$ have the structure of a Harish-Chandra module\footnote{A {\em Harish-Chandra module} for the Harish-Chandra pair $(\mf{g},N)$ is a finitely generated $\mc{U}(\mf{g})$-module with an algebraic action of $N$ such that the differential of the $N$-action agrees with the $\mf{n}$-action coming from the $\mf{g}$-module structure.}. The sheaf $\mc{I}_\mc{O}$ is the {\em standard Harish-Chandra sheaf} associated to the orbit $O$ and the connection $\mc{O}_O$ \cite[Ch. 4 \S 5]{localization}. The goal of this section is to describe the $\mf{g}$-module structure on $\Gamma(X, \mc{I}_O)$ using the formulas (\ref{E_0}) - (\ref{H_infty}). 

We begin with the single point orbit $0$. In general, to describe a  $\mc{D}_t$-module $\mc{F}$, we will describe the $\mc{D}_t(U_i)$-module structure on the vector spaces $\mc{F}(U_i)$ for $i=0,\infty$. However, in this first case, it is sufficient just to work only in the chart $U_0$ because the support of the sheaf $\mc{I}_0$ is contained entirely in the chart $U_0$. (The support of $\mc{I}_0$ is $0$.)  

Let $j_0:0 \hookrightarrow U_0$ be inclusion. Then 
\[
\mc{I}_0(U_0)=j_{0+}(\mc{O}_0)(U_0) = j_{0+}(R(0)).
\]
The ring of regular functions $R(0)$ of the single point variety $0$ is isomorphic to $\C$, as is the ring $D(0)$ of differential operators. The $D(0)$-action on $R(0)$ is left multiplication. 

Because $0$ and $U_0 \simeq \C$ are both affine, we can describe the $D(U_0)$-module structure on the push-forward $j_{0+}(R(0))$ directly using the definition of the $D$-module direct image functor for polynomial maps between affine spaces \cite[Ch. 1 \S 11]{d-modulesnotes}. In the notation of \cite{d-modulesnotes}, we have 
\[
j_{0+}(R(0))=D_{U_0 \leftarrow 0} \otimes_{D(0)}R(0).
\]
Because $D(0)\simeq R(0) \simeq \C$, this is isomorphic to the left $D(U_0)$-module 
\[
D_{U_0 \leftarrow 0}=R(0) \otimes_{R(U_0)}D(U_0). 
\]
Here the $D(U_0)$-module structure on $D_{U_0 \leftarrow 0}$ is given by right multiplication on the second tensor factor by the transpose \cite[Ch. 1 \S5]{d-modulesnotes} of a differential operator $T \in D(U_0)$. As an $R(U_0)$($=\C[z]$)-module, $R(0)$ is isomorphic to $\C[z]/z\C[z]$. Hence the $D(U_0)$-module $j_{0+}(R(0))$ is isomorphic to the $D(U_0)$-module \begin{equation}
    \label{N}
D(U_0)/D(U_0)z = \bigoplus_{i \geq 0} \partial_z^i \delta,
\end{equation}
where $\delta:\C \rightarrow \C, 0 \mapsto 1, x \neq 0 \mapsto 0$ is the Dirac indicator function. 

From this discussion, we see that to describe the $\mf{g}$-module structure on $\Gamma(X, i_{0+}(\mc{O}_0))$, it suffices to compute the actions of the differential operators $E_0, F_0, H_0$ on a basis for the $D(U_0)$-module in (\ref{N}). We choose the basis 
\[
\left\{m_k:=\frac{(-1)^k}{k!} \partial_z^k  \delta\right\}_{k \in \Z_{\geq 0}} 
\]
of (\ref{N}). We include the computation of the $E_0$ action as an example, then record the remaining formulas. Let $k \in \Z_{>0}$. Using the relationships $[\partial_z,z]= \partial_z z - z \partial_z = 1$ and $z \delta = 0$, we compute:
\begin{align*}
   E_0 \cdot m_k &= (z^2 \partial_z - z(t-1)) \frac{(-1)^k}{k!} \partial_z^k \delta \\
    &= \frac{(-1)^k}{k!} z^2 \partial_z^{k+1} \delta - \frac{(-1)^k(t-1)}{k!} z \partial_z^k \delta \\
    &= \frac{(-1)^k}{k!}z(\partial_z^{k+1}z \delta - (k+1) \partial_z^k \delta) + \frac{(-1)^{k+1}(t-1)}{k!}(\partial_z^k z\delta - k \partial_z^{k-1}\delta)  \\
    &= \frac{(-1)^{k+1}(k+1)}{k!}(\partial_z^kz \delta - k \partial_z ^{k-1} \delta) + \frac{(-1)^k(t-1)}{(k-1)!} \partial_z^{k-1} \delta \\
    &= \frac{(-1)^k(k+1)}{(k-1)!} \partial_z^{k-1}\delta + \frac{(-1)^k(t-1)}{(k-1)!}\partial_z^{k-1} \delta \\
    &=( -t-k) m_{k-1}.
\end{align*}
Similar computations for $F_0$ and $H_0$ lead to the following formulas.
\begin{align}
    \label{E_0 on N}
    {\color{red} E_0} \cdot m_k &= (-t-k)m_{k-1}  \text{ for }k \neq 0, {\color{red} E_0} \cdot m_0 = 0 \\
        \label{F_0 on N}
    {\color{blue} F_0} \cdot m_k &= (k+1) m_{k+1} \\ 
        \label{H_0 on N}
    {\color{cyan} H_0} \cdot m_k &= (-t-1-2k)m_k
\end{align}
As in Section \ref{Finite dimensional g-modules}, we can capture this $\mf{g}$-module structure with the picture in Figure \ref{irred verma picture}. From the formulas (\ref{E_0 on N}) - (\ref{H_0 on N}) and Figure \ref{irred verma picture}, we see that for $t\in \Z_{\geq 1}$, $\Gamma(X,\mc{I}_O)$ is an irreducible Verma module of highest weight $-t-1$. 
\begin{figure}
    \centering
\begin{tikzpicture}
\matrix(m)[matrix of math nodes,
row sep=3em, column sep=2.5em,
text height=1.5ex, text depth=0.25ex]
{\cdots&{m_k}&\cdots&{m_2}&{m_1}&m_0 &\\};
\path[->,font=\scriptsize]
(m-1-1) edge [bend right, red] node[below] {$-t-1-k$} (m-1-2)
(m-1-2) edge [bend right, blue]  node[above] {$k+1$} (m-1-1)
(m-1-3) edge [bend right, red] node[below] {$-t-3$}   (m-1-4) 
edge [bend right, blue] node[above] {$k$}  (m-1-2)
(m-1-2) edge [bend right, red] node[below] {$-t-k$}   (m-1-3)
(m-1-2) edge [loop above, cyan] node[above] {$-t-1-2k$} (m-1-2)
(m-1-4) edge [bend right, red]  node[below] {$-t-2$}  (m-1-5)
edge [bend right, blue] node[above] {$3$}  (m-1-3)
(m-1-5) edge[bend right, red]  node[below] {$-t-1$}  (m-1-6)
edge [bend right, blue]  node[above] {$2$} (m-1-4)
(m-1-5) edge [loop above, cyan] node[above] {$-t-3$} (m-1-5)
(m-1-4) edge [loop above, cyan] node[above] {$-t-5$} (m-1-4)
(m-1-6) edge [bend right, blue] node[above] {$1$} (m-1-5)
(m-1-6) edge [loop above, cyan] node[above] {$-t-1$} (m-1-6);
\end{tikzpicture}
    \caption{Action of ${\color{red} E_0}$, ${\color{blue}F_0}$, and ${\color{cyan}H_0}$ on $\mc{I}_0(U_0)$}
    \label{irred verma picture}
\end{figure}

\begin{theorem}
\label{irreducible verma}
 Let $t \in \Z_{\geq 1}$ and $\mc{I}_0$ the standard Harish-Chandra sheaf for $\mc{D}_t$ attached to the closed $N$-orbit $0 \in X$. Then as $\mf{g}$-modules,
\[
\Gamma(X, \mc{I}_0) \simeq M(-t-1), 
\]
where $M(-t-1)$ is the irreducible Verma module of highest weight $-t-1$. 
\end{theorem}

\begin{remark}
\label{singular infinitesimal character}
One can see from an inspection of formulas (\ref{E_0 on N}) - (\ref{H_0 on N}) and Figure \ref{irred verma picture} that Theorem \ref{irreducible verma} also holds when $t=0$. In this setting, the Verma module $M(-1)$ has singular infinitesimal character, so the statement is not an example of Theorem \ref{localization}. (This is why we don't include $t=0$ in the statement of Theorem \ref{irreducible verma}.) 
\end{remark}

Next we examine the standard Harish-Chandra sheaf attached to the open orbit $U_\infty$. Let $i_\infty: U_\infty \hookrightarrow X$ be inclusion. To describe the $\mc{D}_t$-module 
\[
\mc{I}_{U_\infty}:=i_{\infty +}( \mc{O}_{U_\infty}),
\]
we will compute the local $\mf{g}$-module structure of the vector spaces $\mc{I}_{U_\infty}(U_0)$ and $\mc{I}_{U_\infty}(U_\infty)$ using the formulas (\ref{E_0}) - (\ref{H_infty}). In the chart $U_\infty$, we have 
\begin{equation}
    \label{M}
    \mc{I}_{U_\infty}(U_\infty)=R(U_\infty)=\C[w]. 
\end{equation}
We choose the basis 
\begin{equation}
    \label{basis m}
    \left\{ n_k:=(-1)^k w^k \right\}_{k \in \Z_\geq 0}
\end{equation}
of (\ref{M}). The $E_\infty, F_\infty, H_\infty$ actions on the basis (\ref{basis m}) are given by 
\begin{align}
    \label{E_infty on M}
    {\color{red} E_\infty} \cdot n_k &= k n_{k-1} \\
    \label{F_infty on M}
    {\color{blue} F_\infty} \cdot n_k & = ((t-1)-k)n_{k+1}\\ 
    \label{H_infty on M}
    {\color{cyan} H_\infty} \cdot n_k &= ((t-1)-2k) n_k
\end{align}

Hence the $\mf{g}$-module structure of the vector space $\mc{I}_{U_\infty}(U_\infty)$ is given by Figure \ref{picture of M}.

\begin{figure}
\centering
\begin{tikzpicture}
\matrix(m)[matrix of math nodes,
row sep=3em, column sep=2.5em,
text height=1.5ex, text depth=0.25ex]
{\cdots&{n_k}&\cdots&{n_2}&{n_1}&n_0 &\\};
\path[->,font=\scriptsize]
(m-1-1) edge [bend right, red] node[below] {$k+1$} (m-1-2)
(m-1-2) edge [bend right, blue]  node[above] {$t-1-k$} (m-1-1)
(m-1-3) edge [bend right, red] node[below] {$3$}   (m-1-4) 
edge [bend right, blue] node[above] {$t-k$}  (m-1-2)
(m-1-2) edge [bend right, red] node[below] {$k$}   (m-1-3)
(m-1-2) edge [loop above, cyan] node[above] {$t-1-2k$} (m-1-2)
(m-1-4) edge [bend right, red]  node[below] {$2$}  (m-1-5)
edge [bend right, blue] node[above] {$t-3$}  (m-1-3)
(m-1-5) edge[bend right, red]  node[below] {$1$}  (m-1-6)
edge [bend right, blue]  node[above] {$t-2$} (m-1-4)
(m-1-5) edge [loop above, cyan] node[above] {$t-3$} (m-1-5)
(m-1-4) edge [loop above, cyan] node[above] {$t-5$} (m-1-4)
(m-1-6) edge [bend right, blue] node[above] {$t-1$} (m-1-5)
(m-1-6) edge [loop above, cyan] node[above] {$t-1$} (m-1-6);
\end{tikzpicture}
\caption{Action of ${\color{red}E_\infty}$, ${\color{blue} F_\infty}$, and ${\color{cyan} H_\infty}$ on $\mc{I}_{U_\infty}(U_\infty)$}
\label{picture of M}
\end{figure}

It remains to describe the $\mf{g}$-module structure in the other chart $U_0$. Let $k_0:V \hookrightarrow U_0$ and $k_\infty: V \hookrightarrow U_\infty$ be inclusion. Then
\begin{align*}
    \mc{I}_{U_\infty}(U_0)&=k_{0+}k_\infty^+(R(U_\infty))\\
    &=k_{0+}\left(R(V) \otimes_{R(U_\infty)} R(U_\infty)\right)\\
    &=k_{0+} \left( \C[w,w^{-1}] \otimes_{\C[w]} \C[w]\right) \\
    &= k_{0+}\left(\C[w, w^{-1}]\right).
\end{align*}
Because the map $k_0$ is an open immersion, the direct image $k_{0+}(\C[w, w^{-1}])=\C[w, w^{-1}]$ as a vector space, with $D(U_0)$-module structure given by the restriction of the $D(V)$-action to the subring $D(U_0)\subset D(V)$. Hence, as a $D(U_0)$-module, 
\begin{equation}
    \label{M in other chart}
    \mc{I}_{U_\infty}(U_0) = \C[z,z^{-1}].
\end{equation}
We choose the basis 
\[
\left\{ n_k:=(-1)^kz^{-k}\right\}_{k \in \Z}
\]
of (\ref{M in other chart}) to align with the basis (\ref{basis m}) for $\C[w]$ given above. The actions of $E_0, F_0, H_0$ on $\mc{I}_{U_\infty}(U_0)$ are given by the formulas 
\begin{align}
    \label{E_0 on M}
    {\color{red} E_0} \cdot n_k &= (t-1+k)n_{k-1}, \\
    \label{F_0 on M}
    {\color{blue} F_0} \cdot n_k & = -kn_{k+1},\\ 
    \label{H_0 on M}
    {\color{cyan} H_0} \cdot n_k &= (-(t-1)-2k) n_k.
\end{align}
Figure \ref{other picture of M} illustrates this action.
\begin{figure}
\centering
\begin{tikzpicture}
\matrix(m)[matrix of math nodes,
row sep=3em, column sep=2.5em,
text height=1.5ex, text depth=0.25ex]
{\cdots&{n_2}&{n_1}&{n_0}&{n_{-1}}&n_{-2} &\cdots\\};
\path[->,font=\scriptsize]
(m-1-1) edge [bend right, red] node[below] {$t+2$} (m-1-2)
(m-1-2) edge [bend right, blue]  node[above] {$-2$} (m-1-1)
(m-1-3) edge [bend right, red] node[below] {$t$}   (m-1-4) 
edge [bend right, blue] node[above] {$-1$}  (m-1-2)
(m-1-7) edge [bend right, blue] node[above] {$3$}  (m-1-6)
(m-1-2) edge [bend right, red] node[below] {$t+1$}   (m-1-3)
(m-1-2) edge [loop above, cyan] node[above] {$-t-3$} (m-1-2)
(m-1-3) edge [loop above, cyan] node[above] {$-t-1$} (m-1-3)
(m-1-4) edge [bend right, red]  node[below] {$t-1$}  (m-1-5)
(m-1-6) edge[bend right, red]  node[below] {$t-3$}  (m-1-7)
(m-1-5) edge[bend right, red]  node[below] {$t-2$}  (m-1-6)
edge [bend right, blue]  node[above] {$1$} (m-1-4)
(m-1-5) edge [loop above, cyan] node[above] {$-t+3$} (m-1-5)
(m-1-4) edge [loop above, cyan] node[above] {$-t+1$} (m-1-4)
(m-1-6) edge [bend right, blue] node[above] {$2$} (m-1-5)
(m-1-6) edge [loop above, cyan] node[above] {$-t+5$} (m-1-6);
\end{tikzpicture}
\caption{Action of ${\color{red}E_0}$, ${\color{blue} F_0}$, and ${\color{cyan} H_0}$ on $\mc{I}_{U_\infty}(U_0)$}
\label{other picture of M}
\end{figure}

With this information, we can describe the $\mf{g}$-module structure on $\Gamma(X,\mc{I}_{U_\infty})$. A global section of $\mc{I}_{U_\infty}$ is a pair of functions $q(w) \in \C[w]$, $p(z) \in \C[z, z^{-1}]$ such that $q(1/z)=p(z)$. A Lie algebra basis element $X \in \{E,F,H\}$ acts on this global section by
\[
X \cdot (q, p) = (X_\infty \cdot q, X_0 \cdot p),
\]
where the actions of $X_\infty$ and $X_0$ are those given in Figures \ref{picture of M} and \ref{other picture of M} and equations (\ref{E_infty on M}) - (\ref{H_0 on M}). By construction, these actions are compatible on the intersection; that is, on $V$,
\[
X_\infty \cdot q (1/z) = \psi(X_\infty) \cdot p(z)= X_0 \cdot p(z),
\]
where $\psi$ is the ring isomorphism \ref{psi} which defines $\mc{D}_t$.
Because $\C[w] \subset \C[z, z^{-1}]$, a choice of a polynomial in $\C[w]$ uniquely determines a global section, so the space of global sections can be identified with $\C[w]$. Hence, as a $\mf{g}$-module, $\Gamma(X, \mc{I}_{U_\infty}) \simeq \C[w]$ with action as in Figure \ref{picture of M}. We can see from formulas (\ref{E_infty on M}) - (\ref{H_infty on M}) that this $\mf{g}$-module is the dual Verma module of highest weight $t-1$. It has an irreducible $t$-dimensional submodule spanned by $\{n_0, \ldots n_{t-1}\}$. 
\begin{theorem}
 Let $t \in \Z_{\geq 1}$ and $\mc{I}_{U_\infty}$ the standard Harish-Chandra sheaf for $\mc{D}_t$ attached to the open $N$-orbit $U_\infty \subset X$. Then as $\mf{g}$-modules, 
\[
\Gamma(X, \mc{I}_{U_\infty}) \simeq I(t-1),
\]
where $I(t-1)$ is the dual Verma module of highest weight $t-1$. 
\end{theorem}

\section{Admissible representations of $\SL(2,\R)$}
\label{Admissible representations of SL(2,R)}

In Section \ref{Verma modules}, we described the $\mc{D}_t$-modules corresponding to Verma modules and dual Verma modules. These $\mc{D}_t$-modules were the standard Harish-Chandra sheaves for the Harish-Chandra pair $(\mf{g},N)$. In this section, we will describe four $\mc{D}_t$-modules which are constructed in a similar way from $K$-orbits on $X$, where 
\[
K = \left\{ \left. \bp a & 0 \\ 0 & a^{-1} \ep \right\vert a \in \C^\times \right\}. 
\]
These are the standard Harish-Chandra sheaves for the Harish-Chandra pair ($\mf{g},K)$. The group $K$ is the complexification of the maximal compact subgroup of the real Lie group 
\[
\SU(1,1)= \left\{  M \in \SL(2,\C) \left\vert M \bp 1 & 0 \\ 0 & -1 \ep M^* = \bp 1 & 0 \\ 0 & -1 \ep \right. \right\}, 
\]
which is isomorphic to $\SL(2,\R)$. We will see that the global sections of the standard Harish-Chandra sheaves described in this section are the Harish-Chandra modules attached to discrete series and principal series representations of $\SL(2,\R)$. 

The group $K$ acts on $X$ by restriction of the action (\ref{G action on X}). There are three orbits: the two single-point orbits $0$ and $\infty$, and the open orbit $V\simeq \C^\times$. 
\begin{center}
\begin{tikzpicture}[scale=0.4]
  \tikzset{
    dot/.style={circle,inner sep=0.15pt,fill,label={\tiny #1},name=#1},
  }
   \begin{scope}[xshift=-5cm]
\shade[ ball color = red!40, opacity = 0.4] (0,0) circle (2);
\draw[opacity=0, name path=P2] (0,0)--(0,2);
\draw[opacity=0, name path=P4] (0,0)--(0,-2);
\draw[] (0,0) circle (2);
\draw[] (-2,0) arc (180:360:2 and 0.8);
\draw[dashed] (2,0) arc (0:180:2 and 0.8);
\draw[name path=P1, rotate=100] (-2,0) arc (180:360:2 and 1);
\draw[rotate=100,dashed,name path=P3] (2,0) arc (0:180:2 and 1);
\path [name intersections={of=P1 and P2,by=N}]; \node[dot] at
(N) {};
\path [name intersections={of=P3 and P4, by=S}];
\node[dot] at (S) {};

\end{scope}

\draw[fill] (0,0) circle (0.2 and 0.08);
\node at (-1.5,0) {\( = \)};
\node at (1.5,0) {\( \sqcup \)};
\node at (0,-2) {\( 0 \)};
\node at (-5,-3) {$X$};
\node at (5.5,-3) {$V$};
\node at (8.5,0) {\( \sqcup \)};
\draw[fill] (10,0) circle (0.2 and 0.08);
\node at (10,-2) {\( \infty \)};

 \begin{scope}[xshift=5cm]
 \shade[ ball color = red!40, opacity = 0.4] (0,0) circle (2);
\draw[opacity=0, name path=P2] (0,0)--(0,2);
\draw[opacity=0, name path=P4] (0,0)--(0,-2);
\draw[] (0,0) circle (2);
\draw[] (-2,0) arc (180:360:2 and 0.8);
\draw[dashed] (2,0) arc (0:180:2 and 0.8);
\draw[name path=P1, rotate=100] (-2,0) arc (180:360:2 and 1);
\draw[rotate=100,dashed,name path=P3] (2,0) arc (0:180:2 and 1);
\path [name intersections={of=P1 and P2,by=N}]; 
\path [name intersections={of=P3 and P4, by=S}];
\draw[fill=white, very thin] (N) circle (0.2 and 0.08);
\draw[opacity=0.4, very thin] (S) circle (0.2 and 0.08);
\fill[white] (S) circle (0.2 and 0.08);
\draw[fill=red!20,opacity=0.4, very thin] (S) circle (0.2 and 0.08);
\end{scope}
      \end{tikzpicture}
\end{center}

We will begin by describing the standard Harish-Chandra sheaves constructed from the closed orbits. The standard Harish-Chandra sheaf associated to the orbit $0$ is \[
\mc{I}_0:=i_{0+}(\mc{O}_0).
\]
We described the structure of this $\mc{D}_t$-module in Section \ref{Verma modules}. The $\mf{g}$-module structure on its space of global sections is given by formulas (\ref{E_0 on N}) - (\ref{H_0 on N}) and Figure \ref{irred verma picture}.

The standard Harish-Chandra sheaf 
\[
\mc{I}_\infty:=i_{\infty +}(\mc{O}_\infty)
\]
attached to the closed orbit $\infty$ has a similar structure. As it is supported entirely in the chart $U_\infty$, it suffices to describe only $\mc{I}_\infty(U_\infty)$. By analogous arguments to those in Section \ref{Verma modules}, we have 
\[
\mc{I}_\infty(U_\infty) = \bigoplus_{i \geq 0} \partial_w^i \delta,
\]
where $\delta$ is the Dirac indicator function. The actions of $E_\infty, F_\infty$, and $H_\infty$ on the basis 
\[
\left\{ d_k:=\frac{(-1)^k}{k!} \partial_w^k \delta \right\}_{k \in \Z_{\geq 0}}
\]
are given by the formulas 
\begin{align}
    \label{E_infty on D}
    {\color{red} E_\infty} \cdot d_k &= (k+1) d_{k+1}, \\
        \label{F_infty on D}
    {\color{blue} F_\infty} \cdot d_k &= (-t-k) d_{k-1} \text{ for } k>0, \hspace{2mm} {\color{blue} F_\infty} \cdot d_0 = 0, \\
        \label{H_infty on D}
    {\color{cyan} H_\infty} \cdot d_k &=(t+1+2k)d_k. 
\end{align}
Figure \ref{picture of D} illustrates this $\mf{g}$-module structure. We can see that for $t \in \Z_{\geq 1}$, $\Gamma(X, \mc{I}_\infty)$ is an irreducible lowest weight module with lowest weight $t+1$. 

\begin{theorem}
\label{discrete series}
Let $t \in \Z_{\geq 1}$ and $\mc{I}_0, \mc{I}_\infty$ the standard Harish-Chandra sheaves for $\mc{D}_t$ attached to the closed $K$-orbits $0$ and $\infty$, respectively. Then as $\mf{g}$-modules, 
\begin{align*}
    \Gamma(X, \mc{I}_0) &\simeq D_+(-t-1), \text{ and } \\
    \Gamma(X, \mc{I}_\infty) &\simeq D_-(t+1),
\end{align*}
where $D_+(-t-1)$ is the Harish-Chandra module of the (holomorphic) discrete series representation of $\SL(2,\R)$ of highest weight $-t-1$ and $D_-(t+1)$ is the Harish-Chandra module of the (antiholomorphic) discrete series representation of $\SL(2,\R)$ of lowest weight $t+1$. Both of these modules are irreducible. 
\end{theorem}

\begin{remark}
When $t=0$, $\Gamma(X, \mc{I}_0)$ and $\Gamma(X, \mc{I}_\infty)$ are the Harish-Chandra modules of the limits of discrete series representations of $\SL(2,\R)$. As in Remark \ref{singular infinitesimal character}, we exclude this case from the statement of Theorem \ref{discrete series} because it is not an example of Theorem \ref{localization} since the infinitesimal character is singular.
\end{remark}

\begin{figure}
\centering
\begin{tikzpicture}
\matrix(m)[matrix of math nodes,
row sep=3em, column sep=2.5em,
text height=1.5ex, text depth=0.25ex]
{{d_0}&{d_1}&{d_2}&\cdots&{d_k}&\cdots &\\};
\path[->,font=\scriptsize]
(m-1-1) edge [bend right, red] node[below] {$1$} (m-1-2)
(m-1-2) edge [bend right, blue]  node[above] {$-t-1$} (m-1-1)
(m-1-3) edge [bend right, red] node[below] {$3$}   (m-1-4) 
edge [bend right, blue] node[above] {$-t-2$}  (m-1-2)
(m-1-2) edge [bend right, red] node[below] {$2$}   (m-1-3)
(m-1-2) edge [loop above, cyan] node[above] {$t+3$} (m-1-2)
(m-1-1) edge [loop above, cyan] node[above] {$t+1$} (m-1-1)
(m-1-4) edge [bend right, red]  node[below] {$k$}  (m-1-5)
edge [bend right, blue] node[above] {$-t-3$}  (m-1-3)
(m-1-3) edge [loop above, cyan] node[above] {$t+5$} (m-1-3)
(m-1-5) edge[bend right, red]  node[below] {$k+1$}  (m-1-6)
edge [bend right, blue]  node[above] {$-t-k$} (m-1-4)
(m-1-5) edge [loop above, cyan] node[above] { $t+1+2k$} (m-1-5)
(m-1-6) edge [bend right, blue] node[above] {\tiny $-t-1-k$} (m-1-5);
\end{tikzpicture}
\caption{Action of ${\color{red}E_\infty}$, ${\color{blue} F_\infty}$, and ${\color{cyan} H_\infty}$ on $\mc{I}_\infty(U_\infty)$}
\label{picture of D}
\end{figure}

Now we consider the open orbit $V$ and examine the structure of the standard Harish-Chandra sheaf 
\[
\mc{I}_V:=i_{V+}(\mc{O}_V). 
\]
As we did in Section \ref{Verma modules}, we will describe the local $\mf{g}$-module structure of the vector spaces $\mc{I}_V(U_0)$ and $\mc{I}_V(U_\infty)$. To begin, we will record some basic facts about $\mc{D}_t$-module functors. Let 
\begin{center}
    \begin{tikzcd}
    & U_0 \arrow[rd, hookrightarrow, "i_0"]& \\
    V \arrow[ru, hookrightarrow, "j_0"] \arrow[rr, hookrightarrow, "i_V"] \arrow[rd, hookrightarrow, "j_\infty"]&  & X \\
    & U_\infty \arrow[ru, hookrightarrow, "i_\infty"]& 
    \end{tikzcd}
\end{center}
be the natural inclusions of varieties. Because $i_0$ and $i_\infty$ are open immersions, the $\mc{D}_t$-module push-forward functors $i_{0+}$ and $i_{\infty+}$ agree with the sheaf-theoretic push-forward functors. Hence for any $\mc{D}_V$-module $\mc{F}$, 
\begin{equation}
    \label{first reduction}
    i_{V+}(\mc{F})(U_k)=j_{k+}(\mc{F})(U_k)
\end{equation}
for $k=0, \infty$. Moreover, because $j_k$, $k=0, \infty$ are affine immersions, 
\begin{equation}
    \label{second reduction}
    j_{k+}(\mc{F})(U_k)=\mc{F}(U_k)
\end{equation}
as vector spaces, with $D(U_k)$-module structure coming from the natural inclusions $D(U_k) \subset D(V)$. Applying (\ref{first reduction}) and (\ref{second reduction}) to the $\mc{D}_V$-module $\mc{O}_V$, we see that 
\begin{align}
    \label{P locally on U_0}
    \mc{I}_V(U_0)&=\C[z,z^{-1}],\\
    \label{P locally on U_infty}
    \mc{I}_V(U_\infty)&=\C[w, w^{-1}].
\end{align}
We will compute the local $\mf{g}$-module structure on these vector spaces using formulas (\ref{E_0}) - (\ref{H_infty}) and the basis 
\[
\left\{n_k:=(-1)^kw^k = (-1)^kz^{-k}\right\}_{k \in \Z}.
\]
The actions of $E_\infty, F_\infty, H_\infty$ on $\mc{I}_V(U_\infty)$ are given by 
\begin{align}
    \label{E_infty on P}
    {\color{red} E_\infty} \cdot n_k &= k n_{k-1}, \\
    \label{F_infty on P}
    {\color{blue} F_\infty} \cdot n_k & = ((t-1)-k)n_{k+1},\\ 
    \label{H_infty on P}
    {\color{cyan} H_\infty} \cdot n_k &= ((t-1)-2k) n_k.
\end{align}
The actions of $E_0, F_0, H_0$ on $\mc{I}_V(U_0)$ are given by  
\begin{align}
    \label{E_0 on P}
    {\color{red} E_0} \cdot n_k &= (t-1+k)n_{k-1}, \\
    \label{F_0 on P}
    {\color{blue} F_0} \cdot n_k & = -kn_{k+1},\\ 
    \label{H_0 on P}
    {\color{cyan} H_0} \cdot n_k &= (-(t-1)-2k) n_k.
\end{align}
Figures \ref{picture of P} and \ref{other picture of P} illustrate the resulting $\mf{g}$-modules. 

One can see from a brief inspection that the two $\mf{g}$-modules in Figures \ref{picture of P} and  \ref{other picture of P} are isomorphic\footnote{An explicit isomorphism is given by $\mc{I}_V(U_\infty)\rightarrow \mc{I}_V(U_0), n_k \mapsto n_{k-t+1}$.}. Hence the space of global sections of $\mc{I}_V$ can be identified with $\C[w,w^{-1}]$, with $\mf{g}$-module structure as in Figure \ref{picture of P}. 
 
\begin{figure}
\centering
\begin{tikzpicture}
\matrix(m)[matrix of math nodes,
row sep=3em, column sep=2.5em,
text height=1.5ex, text depth=0.25ex]
{\cdots&{n_2}&{n_1}&{n_0}&{n_{-1}}&n_{-2} &\cdots\\};
\path[->,font=\scriptsize]
(m-1-1) edge [bend right, red] node[below] {$3$} (m-1-2)
(m-1-2) edge [bend right, blue]  node[above] {$t-3$} (m-1-1)
(m-1-3) edge [bend right, red] node[below] {$1$}   (m-1-4) 
edge [bend right, blue] node[above] {$t-2$}  (m-1-2)
(m-1-7) edge [bend right, blue] node[above] {$t+2$}  (m-1-6)
(m-1-2) edge [bend right, red] node[below] {$2$}   (m-1-3)
(m-1-2) edge [loop above, cyan] node[above] {$t-5$} (m-1-2)
(m-1-3) edge [loop above, cyan] node[above] {$t-3$} (m-1-3)
(m-1-4)
edge [bend right, blue] node[above] {$t-1$}  (m-1-3)
(m-1-6) edge[bend right, red]  node[below] {$-2$}  (m-1-7)
(m-1-5) edge[bend right, red]  node[below] {$-1$}  (m-1-6)
edge [bend right, blue]  node[above] {$t$} (m-1-4)
(m-1-5) edge [loop above, cyan] node[above] {$t+1$} (m-1-5)
(m-1-4) edge [loop above, cyan] node[above] {$t-1$} (m-1-4)
(m-1-6) edge [bend right, blue] node[above] {$t+1$} (m-1-5)
(m-1-6) edge [loop above, cyan] node[above] {$t+3$} (m-1-6);
\end{tikzpicture}
\caption{Action of ${\color{red}E_\infty}$, ${\color{blue} F_\infty}$, and ${\color{cyan} H_\infty}$ on $\mc{I}_V(U_\infty)$}
\label{picture of P}
\end{figure}

\begin{figure}
\centering
\begin{tikzpicture}
\matrix(m)[matrix of math nodes,
row sep=3em, column sep=2.5em,
text height=1.5ex, text depth=0.25ex]
{\cdots&{n_2}&{n_1}&{n_0}&{n_{-1}}&n_{-2} &\cdots\\};
\path[->,font=\scriptsize]
(m-1-1) edge [bend right, red] node[below] {$t+2$} (m-1-2)
(m-1-2) edge [bend right, blue]  node[above] {$-2$} (m-1-1)
(m-1-3) edge [bend right, red] node[below] {$t$}   (m-1-4) 
edge [bend right, blue] node[above] {$-1$}  (m-1-2)
(m-1-7) edge [bend right, blue] node[above] {$3$}  (m-1-6)
(m-1-2) edge [bend right, red] node[below] {$t+1$}   (m-1-3)
(m-1-2) edge [loop above, cyan] node[above] {$-t-3$} (m-1-2)
(m-1-3) edge [loop above, cyan] node[above] {$-t-1$} (m-1-3)
(m-1-4) edge [bend right, red]  node[below] {$t-1$}  (m-1-5)
(m-1-6) edge[bend right, red]  node[below] {$t-3$}  (m-1-7)
(m-1-5) edge[bend right, red]  node[below] {$t-2$}  (m-1-6)
edge [bend right, blue]  node[above] {$1$} (m-1-4)
(m-1-5) edge [loop above, cyan] node[above] {$-t+3$} (m-1-5)
(m-1-4) edge [loop above, cyan] node[above] {$-t+1$} (m-1-4)
(m-1-6) edge [bend right, blue] node[above] {$2$} (m-1-5)
(m-1-6) edge [loop above, cyan] node[above] {$-t+5$} (m-1-6);
\end{tikzpicture}
\caption{Action of ${\color{red}E_0}$, ${\color{blue} F_0}$, and ${\color{cyan} H_0}$ on $\mc{I}_V(U_0)$}
\label{other picture of P}
\end{figure}
 
\begin{remark}
If $t \in \Z_{\geq 1}$, we can see from Figure \ref{picture of P} that $\Gamma(X, \mc{I}_V)$ has a $t$-dimensional irreducible submodule spanned by $\{n_0, \ldots n_{t-1}\}$, and the quotient of $\Gamma(X, \mc{I}_V)$ by this submodule is isomorphic to the direct sum of the discrete series representations $D_+(-t-1)$ and $D_-(t+1)$ from Theorem \ref{discrete series}.  
\end{remark} 

By construction, the $\mf{g}$-module $\Gamma(X, \mc{I}_V)$ has a compatible action of $K$ which gives it the structure of a Harish-Chandra module for the Harish-Chandra pair $(\mf{g},K)$. Explicitly, if we identify $\Gamma(X, \mc{I}_V)$ with the $\mf{g}$-module in Figure \ref{picture of P}, this action can be obtained by exponentiating the $H_\infty$-action given by formula (\ref{H_infty on P}):
\begin{equation}
    \label{K action}
   k \cdot n_k = a^{t-1-2k} n_k, \text{ where }k=\bp a & 0 \\ 0 & a^{-1} \ep \in K. 
\end{equation}
With this $K$-action, $\C[w, w^{-1}]=R(V)$ obtains the structure of an irreducible $K$-homogeneous $R(V)$-module\footnote{A $K$-homogeneous $R(V)$-module is an $R(V)$-module $M$ with an algebraic action of $K$ such that the action map $R(V) \otimes M \rightarrow M, p \otimes m \mapsto p \cdot m$ is $K$-equivariant.}.
Up to isomorphism, there is exactly one other irreducible $K$-homogeneous $R(V)$-module. 
\begin{theorem}
\label{K-homogeneous connections}
Up to isomorphism, there are exactly two irreducible $K$-homogeneous $R(V)$-modules. 
\end{theorem}
\begin{proof}
Let $M$ be an irreducible $K$-homogeneous $R(V)$-module. As an algebraic representation of $K$, $M$ has a decomposition
\[
M = \bigoplus_{i \in I} M_i
\]
where $I$ is an indexing set of integers and $M_i$ are $K$-subrepresentations such that for $m_i \in M_i$,
\[
k = \bp a & 0 \\ 0 & a^{-1} \ep \in K \text{ acts by } k \cdot m_i = a^i m_i.
\]
Fix some nonzero $m_i \in M_i$. The subspace 
\[
\bigoplus_{n \in \Z} \C z^n \cdot m_i \subseteq M
\]
is stable under the actions of $R(V)$ and $K$ so it forms a $K$-homogeneous $R(V)$-submodule of $M$. Since $M$ is assumed to be irreducible, this submodule must be all of $M$. 

If $i$ is even, then the assignment $m_i \mapsto z^{i/2}$ gives an isomorphism $M \simeq R(V)$ of $K$-homogeneous $R(V)$-modules. 

If $i$ is odd, then $M$ is not isomorphic to $R(V)$ because $R(V)$ contains the trivial representation of $K$ as a subrepresentation and $M$ does not. (Indeed, for any $n \in \Z$, 
\[
k \cdot (z^n \cdot m_i) = (k \cdot z^n) \cdot (k \cdot m_i) = a^{2n+i}(z^n \cdot m_i),
\]
which is not equal to $z^n \cdot m_i$ for $a \neq 1$ because $2n+i$ is odd.)

If $i \neq j$ are both odd, then the assignment $m_i \mapsto z^\frac{i-j}{2} \cdot m_j$ gives an isomorphism 
\[
\bigoplus_{n \in \Z} \C z^n \cdot m_i \xrightarrow{\sim} \bigoplus_{n \in \Z} \C z^n \cdot m_j
\]
of $K$-homogeneous $R(V)$-modules. We conclude that up to isomorphism, there are exactly two irreducible $K$-homogeneous $R(V)$-modules. 
\end{proof}
\begin{remark}
As mentioned in Section \ref{Verma modules}, the irreducible $K$-homogeneous connections on a $K$-orbit $O$ are parameterized by representations of the component group of $\stab_K{x}$ for $x \in O$. When $x \in O=V$, $\stab_{K}x=\pm 1$, so there are two irreducible $K$-homogeneous connections on $V$. Their global sections are the $K$-homogeneous $R(V)$-modules in Theorem \ref{K-homogeneous connections}.  
\end{remark}
We will refer to $R(V)$ as the {\em trivial} $K$-homogeneous $R(V)$-module and the other module, denoted by $P$, as the {\em non-trivial} $K$-homogeneous $R(V)$-module. We can use the non-trivial module $P$ to construct the fourth and final standard Harish-Chandra sheaf for the pair $(\mf{g},K)$. Let
\[
P=\bigoplus_{n \in \Z} P_{2n+1}
\]
be the decomposition of $P$ into irreducible $K$-representations. Fix $p_1 \in P_1$, and let 
\[
p_k:=z^k \cdot p_1.
\]
The set $\{p_k \}_{k \in \Z}$ forms a basis for $P$. The $R(V)$-module $P$ admits the structure of a $D(V)$-module\footnote{Loosely speaking, we can view $P$ as the ring $z^{1/2}\C[z, z^{-1}]$, which explains the $D(V)$-module structure below.} via the action 
\begin{align*}
    \partial_z \cdot p_k &=\left(k+ \frac{1}{2}\right)p_{k-1}, \\
    z \cdot p_k &= p_{k+1}. 
\end{align*}
By construction, $P$ is a $K$-homogeneous $D(V)$-module. Since $V \simeq \C^\times $ is affine, there is a corresponding $K$-homogeneous connection $\tau$. We construct a standard Harish-Chandra sheaf by pushing forward $\tau$ to $X$ using the $\mc{D}_t$-module push-forward:
\[
\mc{I}_{V, \tau}:= i_{V+}(\tau). 
\]
The local $\mf{g}$-module structure on the vector spaces 
\[
\mc{I}_{V, \tau}(U_i) \simeq P
\]
for  $i=0, \infty$ is given by the formulas 
\begin{align}
    \label{E_0 on other P} 
    {\color{red} E_0} \cdot p_k &= (-t+\frac{3}{2}+k)p_{k+1}, \\
    \label{F_0 on other P} 
    {\color{blue} F_0} \cdot p_k&= (-k-\frac{1}{2}) p_{k-1}, \\
    \label{H_0 on other P}
    {\color{cyan} H_0} \cdot p_k &= (-t+2+2k) p_k,\\
    \label{E_infty on other P}
    {\color{red} E_\infty} \cdot p_k &= (k + \frac{1}{2}) p_{k+1}, \\
    \label{F_infty on other P} 
    {\color{blue} F_\infty} \cdot p_k &= (-t-\frac{1}{2}-k)p_{k-1}, \\
    \label{H_infty on other P} 
    {\color{cyan} H_\infty} \cdot p_k &= (2k+t)p_k. 
\end{align}
These $\mf{g}$-modules are illustrated in Figures \ref{picture of other P} and \ref{other picture of other P}. We can see that they are irreducible for any value of $t \in \Z$. 

\begin{theorem}
\label{principal series}
Let $t \in \Z_{\geq 1}$ and $\mc{I}_V$, $\mc{I}_{V, \tau}$ the standard Harish-Chandra sheaves for $\mc{D}_t$ corresponding to the trivial and non-trivial $K$-homogeneous connections on $V$, respectively. Then as $\mf{g}$-modules, 
\begin{align*}
    \Gamma(X, \mc{I}_V) &\simeq P_+(t-1) \\
    \Gamma(X, \mc{I}_{V, \tau}) &\simeq P_{-}(t-1)
\end{align*}
where $P_+(t-1)$, $P_-(t-1)$ are the Harish-Chandra modules associated to the reducible and irreducible (resp.) principal series representations of $\SL(2,\R)$ corresponding to the parameter $\lambda-\rho \in \mf{h}^*$ with $\alpha^\vee(\lambda)=t$. 
\end{theorem}

\begin{remark}
Theorems \ref{discrete series} and \ref{principal series} describe two families of representations which arise in the classification of irreducible admissible representations of $\SL(2,\R)$. A complete geometric classification using standard Harish-Chandra sheaves can be obtained by removing the regularity and integrality conditions on $\lambda$. 
\end{remark}

\begin{figure}
\centering
\begin{tikzpicture}
\matrix(m)[matrix of math nodes,
row sep=3em, column sep=2.5em,
text height=1.5ex, text depth=0.25ex]
{\cdots&{p_{-2}}&{p_{-1}}&{p_0}&{p_1}&p_2 &\cdots\\};
\path[->,font=\scriptsize]
(m-1-1) edge [bend right, red] node[below] {$-5/2$} (m-1-2)
(m-1-2) edge [bend right, blue]  node[above] {$-t+3/2$} (m-1-1)
(m-1-3) edge [bend right, red] node[below] {$-1/2$}   (m-1-4) 
(m-1-4) edge [bend right, red] node[below] {$1/2$}   (m-1-5) 
(m-1-3) edge [bend right, blue] node[above] {$-t+1/2$}  (m-1-2)
(m-1-7) edge [bend right, blue] node[above] {$-t-7/2$}  (m-1-6)
(m-1-2) edge [bend right, red] node[below] {$-3/2$}   (m-1-3)
(m-1-2) edge [loop above, cyan] node[above] {$t-4$} (m-1-2)
(m-1-3) edge [loop above, cyan] node[above] {$t-2$} (m-1-3)
(m-1-4)
edge [bend right, blue] node[above] {$-t-1/2$}  (m-1-3)
(m-1-6) edge[bend right, red]  node[below] {$5/2$}  (m-1-7)
(m-1-5) edge[bend right, red]  node[below] {$3/2$}  (m-1-6)
edge [bend right, blue]  node[above] {$-t-3/2$} (m-1-4)
(m-1-5) edge [loop above, cyan] node[above] {$t+2$} (m-1-5)
(m-1-4) edge [loop above, cyan] node[above] {$t$} (m-1-4)
(m-1-6) edge [bend right, blue] node[above] {$-t-5/2$} (m-1-5)
(m-1-6) edge [loop above, cyan] node[above] {$t+4$} (m-1-6);
\end{tikzpicture}
\caption{Action of ${\color{red}E_\infty}$, ${\color{blue} F_\infty}$, and ${\color{cyan} H_\infty}$ on $\mc{I}_{V,\tau}(U_\infty)$}
\label{picture of other P}
\end{figure}

\begin{figure}
\centering
\begin{tikzpicture}
\matrix(m)[matrix of math nodes,
row sep=3em, column sep=2.5em,
text height=1.5ex, text depth=0.25ex]
{\cdots&{p_{-2}}&{p_{-1}}&{p_0}&{p_1}&p_2 &\cdots\\};
\path[->,font=\scriptsize]
(m-1-1) edge [bend right, red] node[below] {$-t-5/2$} (m-1-2)
(m-1-2) edge [bend right, blue]  node[above] {$3/2$} (m-1-1)
(m-1-3) edge [bend right, red] node[below] {$-t-1/2$}   (m-1-4) 
edge [bend right, blue] node[above] {$1/2$}  (m-1-2)
(m-1-7) edge [bend right, blue] node[above] {$-7/2$}  (m-1-6)
(m-1-2) edge [bend right, red] node[below] {$-t-3/2$}   (m-1-3)
(m-1-2) edge [loop above, cyan] node[above] {$-t-2$} (m-1-2)
(m-1-3) edge [loop above, cyan] node[above] {$-t$} (m-1-3)
(m-1-4) edge [bend right, red]  node[below] {$-t+1/2$}  (m-1-5)
edge [bend right, blue] node[above] {$-1/2$}  (m-1-3)
(m-1-6) edge[bend right, red]  node[below] {$-t+5/2$}  (m-1-7)
(m-1-5) edge[bend right, red]  node[below] {$-t+3/2$}  (m-1-6)
edge [bend right, blue]  node[above] {$-3/2$} (m-1-4)
(m-1-5) edge [loop above, cyan] node[above] {$-t+4$} (m-1-5)
(m-1-4) edge [loop above, cyan] node[above] {$-t+2$} (m-1-4)
(m-1-6) edge [bend right, blue] node[above] {$-5/2$} (m-1-5)
(m-1-6) edge [loop above, cyan] node[above] {$-t+6$} (m-1-6);
\end{tikzpicture}
\caption{Action of ${\color{red}E_0}$, ${\color{blue} F_0}$, and ${\color{cyan} H_0}$ on $\mc{I}_{V,\tau}(U_0)$}
\label{other picture of other P}
\end{figure}

\section{Whittaker modules}
\label{Whittaker modules}

In Sections \ref{Verma modules} and \ref{Admissible representations of SL(2,R)}, we gave examples of $\mc{D}_t$-modules with a Lie group action that was compatible with the $\mc{D}_t$-module structure\footnote{Compatible in the sense that the differential of the group action agrees with the Lie algebra action coming from $\mc{D}_t$.}. These were examples of Harish-Chandra sheaves. In this section, we will give another example of a $\mc{D}_t$-module with a Lie group action, but now the two actions will differ by a character of the Lie algebra. This is an example of a {\em twisted Harish-Chandra sheaf}. Twisted Harish-Chandra sheaves first arose in \cite[Appendix B]{duality} in the study of Harish-Chandra modules for semisimple Lie groups with infinite center, and were later used in \cite{TwistedSheaves} to provide a geometric description of Whittaker modules.  

Let $N$ be as in Section \ref{Verma modules} and $\mf{n}=\Lie{N}$. Fix a Lie algebra morphism $\eta:\mf{n}\rightarrow \C$. Because $\mf{n}$ is spanned by the matrix $E$, this morphism is determined by the image of $E$, which we will also refer to as $\eta$:
\[
\eta:=\eta(E) \in \C.
\]
Our starting place is the $\eta$-twisted connection 
 $\mc{O}_{U_\infty, \eta}$. As a sheaf of rings on $U_\infty$, 
\[
\mc{O}_{U_\infty, \eta}=\mc{O}_{U_\infty},
\]
but the $\mc{D}_{U_\infty}$-module structure is twisted by $\eta$. It suffices to describe this $\mc{D}_{U_\infty}$-module structure on global sections as $U_\infty$ is affine. The $D(U_\infty)$-action on the vector space $\mc{O}_{U_\infty, \eta}(U_\infty)=\C[w]$ is given by 
\begin{equation}
    \label{twisted action}
    \partial_w \cdot w^k = kw^{k-1} - \eta w^k, \hspace{2mm} w \cdot w^k = w^{k+1}.
\end{equation}
\begin{remark}
Alternatively, we could have described this $D(U_\infty)$-module in terms of exponential functions. One can see from formula (\ref{twisted action}) that the module $W=\text{span}\{w^ke^{-\eta w}\}_{k \in \Z_{\geq 0}}$ with $D(U_\infty)$-action 
\[
\partial_w \cdot (w^ke^{-\eta w}) = \frac{d}{dw}(w^k e^{-\eta w})
\]
is isomorphic to the $D(U_\infty)$-module described above.
\end{remark}

We can use the $\mc{D}_t$-module direct image functor to push the sheaf $\mc{O}_{U_\infty, \eta}$ forward to a $\mc{D}_t$-module on $X$. Let $i_\infty:U_\infty \hookrightarrow X$ be inclusion. Define 
\begin{equation}
    \label{twisted HC sheaf} 
    \mc{I}_{\infty, \eta}:= i_{\infty+}(\mc{O}_{U_\infty, \eta}).
\end{equation}
The $\mc{D}_t$-module $\mc{I}_{\infty, \eta}$ is the {\em standard $\eta$-twisted Harish-Chandra sheaf} \cite[\S3]{TwistedSheaves} associated to the Harish-Chandra pair $(\mf{g},N)$ and the open orbit $U_\infty$. 

\begin{remark} For $\eta \neq 0$, there is no standard $\eta$-twisted Harish-Chandra sheaf associated to the closed orbit $0$.
Indeed, as $\mc{O}_0$ and $\mc{D}_0$ are sheaves on a single point\footnote{Recall that $\mc{O}_0$ is the structure sheaf on the single point $N$-orbit $0 \in X$ and $\mc{D}_0$ is the sheaf of differential operators on $0$ (see set-up in \S\ref{Verma modules}).}, they are simply the data of a vector space, the vector space $\C$. A $\mc{D}_0$-module structure on $\mc{O}_0$ is an action of $\C$ on itself. As the Lie algebra element $E$ is nilpotent, it must act as multiplication by $0$ in any such $\mc{D}_0$-module structure on $\mc{O}_0$. Moreover, as $N$ fixes $0$, in the $\mf{n}$-module structure on $\mc{O}_0$ coming from the action of $N$ on $X$, the Lie algebra element $E$ also acts as multiplication by $0$. Hence for $\eta \neq 0$, it is not possible to give the sheaf $\mc{O}_0$ the structure of a $\mc{D}_0$-module in such a way the $\mf{n}$-module structure coming from the $\mc{D}_0$-action differs by $\eta$ from the $\mf{n}$-module structure coming from the $N$-action, as we have done above for the structure sheaf $\mc{O}_{U_\infty}$ on the open orbit $U_\infty$.
\end{remark}

As we did in Section \ref{Verma modules}, we will describe the $\mc{D}_t$-module structure of $\mc{I}_{\infty, \eta}$ by computing the local $\mf{g}$-module structure on the vector spaces $\mc{I}_{\infty, \eta}(U_\infty)$ and $\mc{I}_{\infty, \eta}(U_0)$ using formulas (\ref{E_0}) - (\ref{H_infty}). From these local descriptions we can identify the $\mf{g}$-action on $\Gamma(X, \mc{I}_{\infty, \eta})$.

In the chart $U_\infty$ we have 
\begin{equation}
    \label{W}
    \mc{I}_{\infty, \eta}(U_\infty)=R(U_\infty)=\C[w]
\end{equation}
as a vector space with $D(U_\infty)$-action given by (\ref{twisted action}). The action of $E_\infty$, $F_\infty$, and $H_\infty$ on the basis
\[
\left\{ n_k:=(-1)^k w^k \right\}_{k \in \Z_{\geq 0}}
\]
is given by the formulas 
\begin{align}
    \label{E_infty on W}
    {\color{red} E_\infty} \cdot n_k &= k n_{k-1} + \eta n_k, \\
    \label{F_infty on W}
    {\color{blue} F_\infty} \cdot n_k &= (t-1-k) n_{k+1} - \eta n_{k+2},\\
    \label{H_infty on W}
    {\color{cyan} H_\infty } \cdot n_k &= (t-1-2k) n_k - 2 \eta n_{k+1}. 
\end{align}
\begin{remark}
Unlike the previous examples, the differential operators $E_\infty, F_\infty$ and $H_\infty$ do not map basis vectors to scalar multiples of other basis vectors. One might think that this is the result of a poor choice of basis, but this is not the case. In fact, we can not choose a basis for this module such that each operator $E_\infty, F_\infty$ and $H_\infty$ maps basis vectors to scalar multiples of basis vectors. Instead of being disappointed by this more complicated module structure, we should use it as evidence that our previous examples were unusually well-behaved. 
\end{remark}
As we did in Sections \ref{Finite dimensional g-modules}, \ref{Verma modules}, and \ref{Admissible representations of SL(2,R)}, we can capture formulas (\ref{E_infty on W}) - (\ref{H_infty on W}) in a picture, see Figure \ref{picture of W}. In this figure, colored arrows now represent linear combinations of basis elements. For example, the two blue arrows labeled $t+2$ and $-\eta$ emanating from $n_1$ represent the relationship $F_\infty \cdot n_1 = (t+2)n_2 -\eta n_3$. 

\begin{figure}
    \centering
    \includegraphics[width=0.8\textwidth]{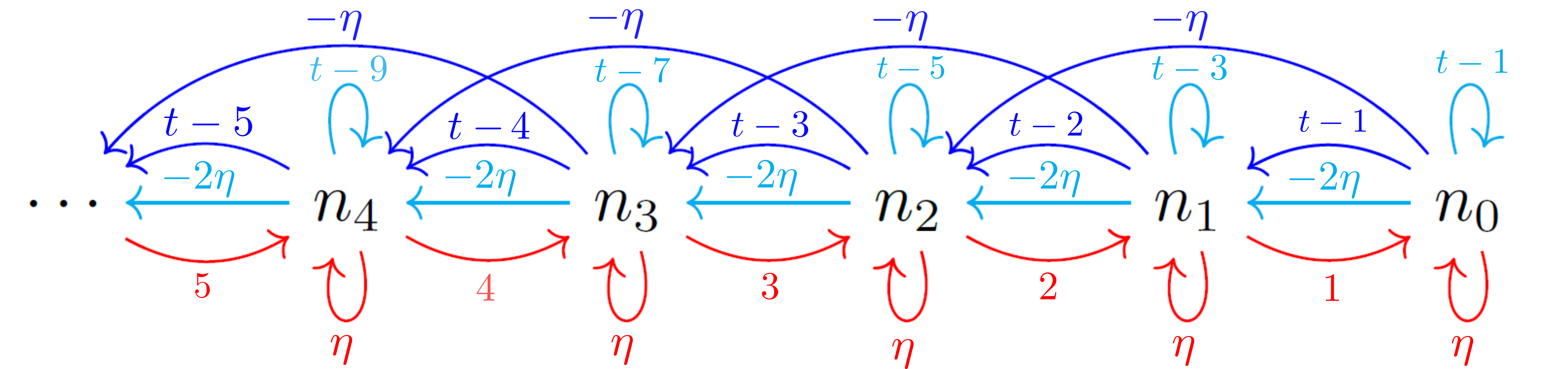}
    \caption{Action of ${\color{red}E_\infty}$, ${\color{blue} F_\infty}$, and ${\color{cyan} H_\infty}$ on $\mc{I}_{\infty, \eta}(U_\infty)$}
    \label{picture of W}
\end{figure}

Now we turn our attention to the other chart $U_0$. By the same arguments as in Section \ref{Verma modules}, as a vector space
\begin{equation}
    \label{W in other chart}
    \mc{I}_{\infty, \eta}(U_0) \simeq \C[z, z^{-1}],
\end{equation}
with $\mc{D}(U_0)$-module structure\footnote{This action is derived from the relationship (\ref{partial relationship}) between $\partial_z$ and $\partial_w$ on $V$.} given by 
\begin{equation}
    \label{twisted action in other chart}
    \partial_z \cdot z^k = kz^{k-1} + \eta z^{k-2}, \hspace{2mm} z \cdot z^k = z^{k+1}
\end{equation}
for $k \in \Z$. The actions of $E_0, F_0$ and $H_0$ on the basis 
\[
\left\{ n_k:= (-1)^k z^{-k} \right\}_{k \in \Z}
\]
are given by the formulas 
\begin{align}
    \label{E_0 on W}
    {\color{red} E_0} \cdot n_k &= (t-1+k)n_{k-1}+ \eta n_k, \\
    \label{F_0 on W}
    {\color{blue} F_0} \cdot n_k & = -kn_{k+1} - \eta n_{k+2},\\ 
    \label{H_0 on W}
    {\color{cyan} H_0} \cdot n_k &= (-(t-1)-2k) n_k - 2 \eta n_{k+1}.
\end{align}
Figure \ref{other picture of W} illustrates this action. 

\begin{figure}
\centering
    \includegraphics[width=1\textwidth]{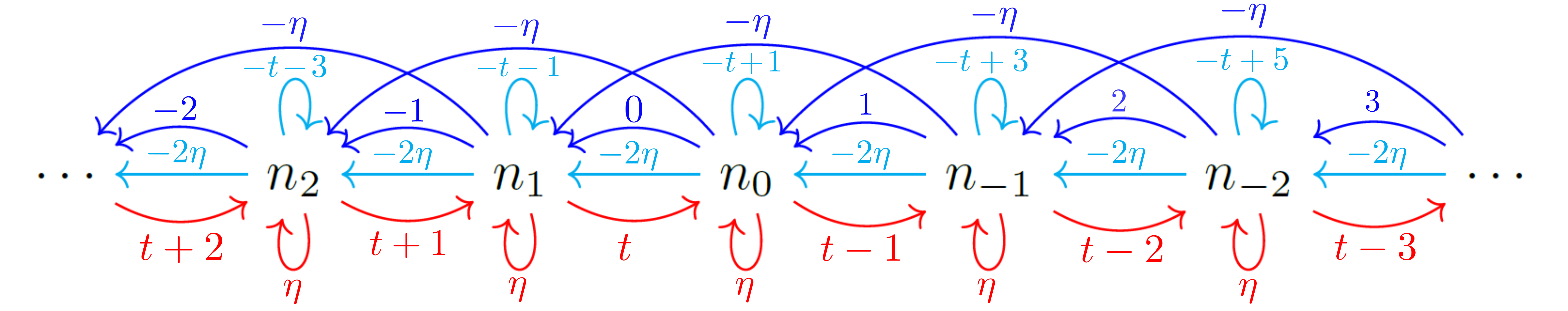}
\caption{Action of ${\color{red}E_0}$, ${\color{blue} F_0}$, and ${\color{cyan} H_0}$ on $\mc{I}_{\infty, \eta}(U_0)$}
\label{other picture of W}
\end{figure}

As we did in Section \ref{Verma modules}, we can use Figures \ref{picture of W} and \ref{other picture of W} and formulas (\ref{E_infty on W})-(\ref{H_0 on W}) to describe the $\mf{g}$-module structure on $\Gamma(X, \mc{I}_{\infty, \eta})$. A global section of $\mc{I}_{\infty, \eta}$ is a pair $(q(w), p(z))$ with $q(w) \in \C[w]$ and $p(z) \in \C[z,z^{-1}]$ such that $q(1/z)=p(z)$, so as a $\mf{g}$-module, $\Gamma(X, \mc{I}_{\infty, \eta})$ is isomorphic to $\C[w]$ with $E,F$, and $H$ actions given by formulas (\ref{E_infty on W})-(\ref{H_infty on W}). 

\begin{remark}
We draw the reader's attention to several properties of the $\mf{g}$-module $\Gamma(X, \mc{I}_{\infty, \eta})$ which can be seen from careful examination of Figures \ref{picture of W} and \ref{other picture of W}:
\begin{enumerate}
    \item The module is generated by the vector $n_0$, which has the property that $E$ acts by a scalar:
    \[
    E \cdot n_0 = \eta n_0.
    \]
    For a general semisimple Lie algebra $\mf{g}$, a vector in a $\mf{g}$-module on which $\mf{n}$ acts by a character is called a {\em Whittaker vector}. A $\mf{g}$-module which is cyclically generated by a Whittaker vector is a {\em Whittaker module}. Hence $\Gamma(X, \mc{I}_{\infty, \eta})$ is a Whittaker module.
    \item The module is irreducible. 
    \item If $\eta=0$, 
    \[
    \mc{I}_{\infty, \eta} = \mc{I}_{U_\infty}
    \]
    is the standard Harish-Chandra sheaf attached to the open orbit $U_\infty$ whose structure we described in Section \ref{Verma modules}.
    \item If $\eta \neq 0$ then the basis $\{n_k\}$ is not a basis of eigenvectors of $H$. In fact, it is impossible to choose a basis of $H$-eigenvectors for the module because it is not a weight module.
    \item The Casimir operator $\Omega = H^2 + 2EF + 2FE \in Z(\mf{g})$ acts on the module\footnote{In fact, $\Omega$ acts on the global sections of any $\mc{D}_t$-module by $(t-1)^2+2(t-1)$ by construction. This can be checked in each chart with a quick computation using (\ref{E_0}) - (\ref{H_infty}) and the relationshps $[\partial_z, z]=[\partial_w,w]=1$. The integer $(t-1)^2+2(t-1)$ is the image of $\Omega$ under the infinitesimal character $\chi:Z(\mf{g})\rightarrow \C, z\mapsto (\lambda - \rho)(p_0(z))$, where $p_0$ is the Harish-Chandra homomorphism.} by $(t-1)^2 + 2(t-1)$.
\end{enumerate}
\end{remark}
If $\eta \neq 0$, the module $\Gamma(X, \mc{I}_{\infty, \eta})$ is an irreducible nondegenerate\footnote{For a general semisimple Lie algebra $\mf{g}$, a character of $\mf{n}$ is {\em nondegenerate} if it is non-zero on all simple root subspaces. For $\mf{sl}(2,\C)$, all non-zero $\mf{n}$-characters are nondegenerate because there is only one simple root.} Whittaker module. Irreducible nondegenerate Whittaker modules were introduced and classified in \cite[\S 3]{Kostant}.

\begin{theorem}
 Let $t \in \Z_{\geq 1}$, $\eta \in \C$ and $\mc{I}_{\infty, \eta}$ the standard $\eta$-twisted Harish-Chandra sheaf for $\mc{D}_t$. Then as $\mf{g}$-modules, 
\[
\Gamma(X, \mc{I}_{\infty, \eta}) \simeq Y(\eta, t),
\]
where $Y(\eta, t)$ is the irreducible $\eta$-Whittaker module of infinitesimal character $\chi: \Omega \mapsto (t-1)^2 + 2(t-1)$. 
\end{theorem}

\bibliographystyle{alpha}
\bibliography{Dubrovnik}

\end{document}